\documentclass[11pt,letterpaper]{article}
\usepackage{amsmath,amssymb,amsfonts,amsthm,euscript,changebar,multicol,amscd}
\usepackage{graphicx,wrapfig,subfigure,sidecap,epsfig,epstopdf}
\usepackage[table,dvipsnames]{xcolor}
\usepackage{hyperref}
\usepackage{url}
\usepackage[small,bf]{caption}
\usepackage{lastpage}
\usepackage[numbers,sort&compress]{natbib}
\usepackage{multirow}
\usepackage{cancel}
\usepackage{siunitx}
\usepackage[normalem]{ulem}
\usepackage{array}
\usepackage{varwidth}
\usepackage{lineno}
\usepackage{subcaption}    % for modern subfigures
\usepackage{hyperref}
\usepackage{soul}

\usepackage[dvipsnames]{xcolor}
\usepackage{amsthm}
\newtheorem{theorem}{Theorem}
\newtheorem{remark}{Remark}
\newtheorem{lemma}[theorem]{Lemma}
\title{A modified static PDE-based approach for multidimensional extrapolation}
\author{Hwi Lee \footnotemark[1]}

\begin{document}
\date{}
\maketitle
\renewcommand{\thefootnote}{\fnsymbol{footnote}}
\footnotetext[1]{ Department of Mathematics, New York Institute of Technology,
Old Westbury, NY 11568 ({\tt E-mail: hlee50@nyit.edu}) }

\begin{abstract}

We present a modified static PDE-based extrapolation method that builds on the approach of Aslam \cite{aslam2014static} for extending smooth fields across interfaces implicitly defined as zero level sets. Our approach introduces the idea of relaxed upwinding finite differences as a key modification within the fast sweeping method. Unlike existing approaches, we require function values and their normal derivatives \emph{only} on the first grid layer inside the domain adjacent to the boundary, thereby improving computational efficiency and flexibility. To further improve accuracy, we introduce a simple boundary reconstruction technique that significantly reduces the numerical error in the extrapolated solution near the boundary. Numerical experiments indicate enhanced performance of our approach across a range of domain geometries, including those with high curvature features.

\end{abstract}

\section{Introduction}

In numerous scientific computing applications, the importance of accurately and efficiently extrapolating known function values outside their known regions is well documented. Specific instances include multiphase flows \cite{lepilliez2016two,gibou2007level}, image processing \cite{vese2002multiphase}, electrodiffusion \cite{mirzadeh2014conservative}, self-assembly of polymers \cite{ouaknin2016self}, to name only a few. In some nonlocal models such as peridynamics \cite{silling2000reformulation}, local Taylor series based extrapolation remains a common modeling choice for imposing boundary conditions, despite the nonlocal models being viewed as alternatives to partial differential equations based models for singular solutions. In addition, as a numerical subroutine,  extrapolation is essential to global numerical PDE solvers such as \cite{ruuth2008simple, kublik2013implicit} where the closest point mapping relies on constant extrapolation. 

Our present work is in the context of level set methods \cite{osher1988fronts}, which extensively employ extrapolation to construct so-called ghost point values \cite{gibou2007level}. The flexibility of the level set framework for interface problems depends on the construction of ghost values that are sufficiently smooth and consistent with the prescribed boundary conditions. Assuming a smooth interface, this can be achieved by extrapolating along normal directions, which effectively reduces the construction to a one dimensional process. To this end, Aslam \cite{aslam2004partial} proposed one of the most frequently used numerical methods, in which a sequence of linear transport equations is solved until steady state is reached. A variant of this time dependent approach is demonstrated in \cite{bochkov2020pde} to handle interfaces with kinks and high curvature. As computationally efficient alternatives, static PDE based approaches have been studied in  \cite{aslam2014static,mccaslin2014fast}, where the steady state equations are solved iteratively using the fast sweeping  \cite{zhao2005fast} and fast marching methods \cite{sethian1999fast}. We also note an implicit static approach in \cite{moroney2017extending}, based on solutions of biharmonic equations. 

We adopt the fast sweeping method approach here due to its simplicity and robustness. Our contributions consist of modifications designed to further improve efficiency and flexibility. First, we introduce the idea of relaxed upwinding finite differences which remain compatible with the fast sweeping method, although they permit some downwind differences. A key feature of our approach is its reduced data requirement: function values and normal derivatives need to be available only on the first grid layer inside the domain adjacent to the boundary. This is consistent with the structure of the analytic normal extension, where boundary data alone determines the smooth extension of function values off the interface. Beyond computational efficiency, limited data availability is intrinsic to some problem settings, for example, numerical solutions of PDEs on surfaces of co-dimesion one. We focus in this work on spatial extrapolation, but the reduced data requirement may also be useful for temporal extrapolation.

The simplest setting of scalar linear hyperbolic equations allows us to clearly demonstrate the effectiveness of the proposed relaxed upwinding as a second order discretization that is applied only locally near the boundary. While the corresponding finite difference stencil involves only a single layer of immediate interior nodes, it is not compact as it reaches further into exterior nodes away from the interface. To ensure global stability, we apply standard second order upwinding at exterior nodes that are not adjacent to the interface.  Compared to the second order compact upwinding method in \cite{benamou2010compact}, we do not require estimation of interface curvature or strong regularity assumptions on the source term; the former requirement is particularly relevant in the presence of sharp corners. In addition, no division of the local mesh into subregions is needed in our formulation, which simplifies the implementation.

Another distinguishing feature of our work is the introduction of a boundary reconstruction technique that reduces the numerical errors  of the computed solutions near the interface. Once the ghost layers are populated with numerically extrapolated values, we consider the problem of reconstructing interior values, deliberately treating them as unknowns. This is motivated by  \cite{moroney2017extending}, where the interpolation of interior values is used as a design perspective for their extrapolation technique. In contrast, we use the interior values as reference data to improve the exterior data that serve as the input to a carefully designed convolution operator. Radial symmetry is enforced in the design of the convolution, with the goal of reducing tangential noise that may be introduced by normal-based extrapolation. The construction of our reconstruction technique is guided primarily by accuracy considerations, while carefully accounting for the associated computational cost. We demonstrate, despite requiring only a small proportion of the computational cost of the original extrapolation procedure, the proposed reconstruction substantially reduces numerical errors near the boundary of the known region. In particular, when combined with quadratic extrapolation, one may obtain fourth order accuracy without computing third order normal derivatives at the boundary, which are otherwise needed for direct cubic extrapolation. We further show that the method is robust for domains with geometric singularities, provided that the underlying solution is smooth and the singularities are sufficiently mild.

The remainder of the paper is organized as follows. The proposed method is presented in detail in Section \ref{sec:2}. Numerical experiments illustrating its performance are presented in Section \ref{sec:3}. Finally, concluding remarks are provided in Section \ref{sec:4}.

\section{Numerical Method}
\label{sec:2}

We assume that a smooth domain $\Omega \subset \mathbb{R}^2$ is implicitly defined by a level set function $\phi$ with $\phi < 0$ in $\Omega$, and $\phi > 0$ in $\mathbb{R}^2 \backslash \Omega$;  a common choice of $\phi$ is the signed distance function \cite{osher2003signed}. The outward normal vector ${\mathbf{n}}$ is then given by
$$
{\mathbf{n}} = \frac{\nabla \phi}{\left|\nabla \phi \right|}
$$
which we compute in this work by the second order centered finite differences. For notational convenience, let us write $n^x$ and $n^y$ for the x- and y- components of ${\mathbf{n}}$, respectively.

Let us recall Aslam's static approach \cite{aslam2014static} to quadratic extrapolation which solves a series of stationary PDEs in the following sequential order
\begin{align}
    \mathbf{n}\cdot \nabla u_{nn} &= 0 \label{eq:first}\\
    \mathbf{n}\cdot \nabla u_{n} &= u_{nn} \label{eq:second}\\
    \mathbf{n}\cdot \nabla u &= u_n \label{eq:third}
\end{align}
in $\mathbb{R}^2 \backslash \Omega$ with the boundary data $u_{nn} = {\mathbf{n}} \cdot \nabla ({\mathbf{n}} \cdot \nabla u )$ and $u_{n} = {\mathbf{n}} \cdot \nabla u$ on $\partial \Omega$. The time dependent versions of the PDEs (\ref{eq:first}) - (\ref{eq:third}) are considered in their earlier work \cite{aslam2004partial}. We should note that for interfaces with high curvature, a modification of the static PDE based approach is recently proposed in \cite{bochkov2020pde} where the full tensors of Cartesian derivatives are extended in lieu of the normal derivatives. 

For a numerical solution of the static PDEs (\ref{eq:first}) - (\ref{eq:third}), we consider a uniform Cartesian mesh $\{(ih,jh)\}$ $(i,j)\in \mathbb{Z}^2$, where $h$ is the discretization parameter. We write $(i,j)$ to denote the grid point $(x_i,y_j)$ and classify each grid point according to its proximity to the interface. We say $(i,j)$ belongs to the first inner layer if $\phi_{i,j}  \leq 0$ but one of its four immediate neighbors $(i-1,j),(i+1,j),(i,j-1),(i,j+1)$ belongs to $\mathbb{R}^2 \backslash \Omega$. The second inner layer is defined as the set of $(i,j)$, not in the first inner layer, with $\phi_{i,j} < 0$ and one of its immediate four neighbors belonging to the first inner layer. We analogously then define the first and second outer layers.

\subsection{Least Squares Approximation of Normal Derivatives}
It is common practice in the extrapolation literature to apply centered finite difference approximations of $({\mathbf{n}} \cdot \nabla)^p u$, $p = 1, 2, \cdots$. This, in turn, means that the computed derivatives are available only on inner layers that recede progressively farther from the boundary as $p$ increases. Instead, we compute all the required normal derivatives simultaneously on the first inner layer, the outermost layer where function values are available, using a local least squares quadratic approximation.

The local least squares linear fitting in \cite{wang2019back} is based on a centered stencil involving the center grid point and its four immediate neighbors. For the quadratic case, we use the upwind 9-point stencil in \cite{coco2013finite}, which is used therein to construct a biquadratic interpolation. Rather than the exact interpolation, we instead employ a least squares quadratic fit for improved stability. More specifically, for ${n}^{x}_{i,j} > 0$ and ${n}^{y}_{i,j} > 0$, we use the points $(i+k, j+l)$, where $(k,l) \in \{0,-1,-2\}^2$. One can verify that the smallest singular value of the resulting regression matrix is bounded below by $Ch^2$ on a uniform rectangular grid. It then follows from \cite{trefethen2019approximation} that our approximation is second order accurate. For more complex boundary geometries, one may need to use a curvature dependent upwind stencil, as in \cite{gabbard2024high}.

We emphasize that the normal derivatives are computed and stored only on the first inner layer, thanks to our relaxed upwinding described in Section~\ref{sec:update}. This not only reduces the number of required computations and the depth of required interior data access, but also leads to lower memory usage, which is particularly relevant for tensor-based extrapolation \cite{bochkov2020pde}.

\subsection{Fast Sweeping Method}
The fast sweeping method is  an efficient iterative method for solving a class of Hamilton-Jacobi equations which, in their simplest forms, reduce to the linear hyperbolic PDEs (\ref{eq:first})-(\ref{eq:third}). The crux of the fast sweeping method is its enforcement of causality by upwind finite differences and Gauss-Seidel iterations with alternating sweeping orders. For a survey of the method and its full-fledged developments, we refer to \cite{zhao2016fast} and the references cited therein. We note that the one-pass type algorithms like the fast marching method \cite{sethian1999fast} are viable alternatives to the fast sweeping method.

In this work, we consider extrapolating onto the narrow band region of $O(h)$, i.e. $|\phi| = O(h)$, as it is often done in existing studies. The update equations in \ref{sec:update} are then iteratively applied in the following alternating orderings
\begin{equation}
\begin{aligned}
\text{(i) } i = 1:I, j = 1:J, & \quad \text{(ii) } i = 1:I, j = J:1 \\
\text{(iii) } i = I:1, j = 1:J, & \quad \text{(iv) } i = I:1, j = J:1
\end{aligned}
\label{eq:sweep}
\end{equation}

to cover all directions of information arrival. In light of utilizing only a single inner layer of solution data, we seek a simpler second order upwinding differencing in the first outer layer than the one proposed in \cite{benamou2010compact}.

\subsection{Upwind second order finite differences}
\label{sec:update}
For simplicity, let us present the one dimensional case for a grid point $x_i$ with $n^{x}_i>0$. The extension to two dimensional is straightforward direction-by-direction. We first assume $x_i$ belongs to the second outer layer and  follow the work \cite{aslam2014static} for a local update equation for $u_i$  given by
\begin{equation}
u_i = \frac{f_i+ n^{x}_i (2u_{i-1}-\frac{1}{2}u_{i-2})/h}{3n_i/(2h)}
\label{eq:st_uw_update}
\end{equation} where $f_i$ denotes the forcing terms on the right hand sides of (\ref{eq:first})-(\ref{eq:third}). This can be derived by applying the standard second order upwinding
\begin{equation}
(u_x)^{+}_i = \frac{3u_i-4u_{i-1} + u_{i-2}}{2h}.
\label{eq:st_uw}
\end{equation}

Our next case is when $x_i$ belongs to the first outer layer, for which we look for an approximation of
$$
(u_x)_i \approx \frac{1}{h}\sum^{2}_{j=-1} \alpha_j u_{i+j}, \qquad \alpha_j \in \mathbb{R}.
$$ In light of (\ref{eq:st_uw}), the weights $\alpha_j$ should be chosen to yield the second order accuracy. This in turn allows us to rewrite the right hand side as a linear combination of upwind, downwind and diffusive terms
\begin{equation*}
\frac{1}{h} \left( w (u_{i}-u_{i-1}) + (1-w) (u_{i+1}-u_{i}) + \left(w-\frac{1}{2}\right) (u_{i+2}-2 u_{i+1} + u_{i})\right)
\end{equation*}
in terms of a single parameter $w\in \mathbb{R}$. While $w=\frac{1}{3}$ corresponds to the maximal third order accuracy, we choose $w=1$ to annihilate the downwind differencing for the sake of causality. We hence obtain \emph{relaxed} upwinding
\begin{equation}
(u_x)^{R+}_i = \frac{1}{h} \left((u_{i}-u_{i-1}) + \frac{1}{2} (u_{i+2}-2 u_{i+1} + u_{i})\right) =  \frac{-2u_{i-1} + 3 u_{i} -2 u_{i+1} + u_{i+2}}{2h},
\label{eq:st_ruw}
\end{equation} consequently, the following update equation
\begin{equation}
u_i = \frac{f_i+ n^{x}_i (u_{i-1} + u_{i+1} - \frac{1}{2}u_{i+2})/h}{3n_i/(2h)}.
\label{eq:relax_uw}
\end{equation}
Let us make a couple of remarks. The diffusive contribution in \ref{eq:st_ruw} is deliberately centered at $i+1$ to avoid restriction to the unique second order centered approximation involving only $i-1,i,i+1$. For a 2D grid point $(i,j)$ in the first outer layer, it may be that only $(i-1,j)$ belongs to the first inner layer while $(i,j-1)$ may belong to an outer layer. In such cases, the full and relaxed upwind second order approximations are used for $y$ and $x$ derivatives, respectively. 

The presence of downwind terms in our relaxed upwinding scheme warrants a careful convergence analysis of the resulting fast sweeping method. We focus on the simplest possible $1D$ setting, for which a proof is no longer trivial, unlike in the case of the fast sweeping method based on (\ref{eq:st_uw_update}).

\begin{lemma}
Let $u(x) $ be the solution to 
$$
u^\prime(x) = f(x), \qquad x \in [0,1]
$$ subject to the boundary condition $u(0) = a, \quad a \in \mathbb{R}$. Let $(u^h)_i$ be the corresponding numerical solution at  $x_i := h{i}$, $i=1,\dots, n, \quad h = \frac{1}{n}$, obtained by applying the fast sweeping method to $$A^hu^h = f$$ where $A^h \in \mathbb{R}^{n \times n}$ is given by 
\[
\frac{1}{2h}
\begin{bmatrix}
3 & -2 & 1 \\
-4 & 3 & 0 \\
1 & -4 & 3 \\
& \ddots & \ddots & \ddots \\
&& 1 & -4 & 3
\end{bmatrix}.
\]
If $u^{h}_{m}$ is the $m$-th iterate after each full sweep (one forward and one backward sweep), then $u^{h}_{m} \to u^{h}$ as $m \to \infty$ for any initial guess $u^{h}_{0}.$

\begin{proof}
    Let $e^{h}_{m} = u^{h}_m - u^h$. Then,  
    \begin{eqnarray*}
    {e}^{h}_{m+1} &=& \underbrace{(-(D+U)^{-1} L)}_{B} \underbrace{(-(D+L)^{-1} U)} _{F} e^{h}_{m}
    \end{eqnarray*} where $B$ and $F$ denote backward and forward sweep operators. Here $D,L,U$ are the diagonal, strictly lower and strictly upper parts of $2h A^h$, respectively.
    
   We re-write the rank-1 matrix $U$ as $U = e_1 r^T$, where $e_1 = [1,0,\dots,0]^T$ and $r = [0,-2,1,0,\dots, 0]^T$. Then,  direct calculation shows 
    $$
    \rho(F) = \rho(-v r^T) = |r^T v| = |-2 v_2 + v_3|  = \frac{11}{27} < 1  $$
    where $v = \left[\frac{1}{3},\frac{4}{9},\frac{13}{27}, \dots, v_n\right]^T$ solves $(D+L)v = e_1$. Now, if we let $w= Bv$ so that 
    $
    BF = -wr^T
    $, it follows 
    $$
    \rho(BF) = \rho(-wr^T) = |r^T w| = |-2 w_2 + w_3|  =  \frac{11}{27} < 1
    $$
    where the last equality is due to solving the first three equations of $(D+U)w = -Lv$ to obtain $w = \left[ \frac{11}{81},\frac{4}{9},\frac{13}{27}, \dots, w_n \right]^T.$
\end{proof}

\end{lemma}

\subsubsection{Reconstruction near boundary}
Cubic and higher order extrapolations can be obtained by suitably extending higher order normal derivatives, but this requires stronger regularity assumptions on the domain boundary, as well as higher order discretizations for both the derivatives and the update equations. For computational efficiency, we instead seek a simple post-processing strategy that can substantially reduce numerical errors in already computed solutions near the domain boundary. This localized approach is well aligned with the popular ghost fluid method \cite{fedkiw1999non} and its subsequent developments, where extrapolated values are typically needed only in the first few outer layers.

Central to our proposed technique is the observation that the convenience of numerical extrapolation along the normal direction may introduce numerical errors in the tangential directions, particularly due to the skewed stencil associated with the relaxed upwind discretization. We hence seek an isotropic smoothing strategy. To this end, we consider the reconstruction of the \emph{known} inner layer function values from the computed outer layer values obtained through the extrapolation procedure.

We look for a discrete kernel $W = [w_{k,l}]_{1 \leq k,l \leq r}$, with $r$ odd, so that
$$
u_{i,j} = \sum_{k =1}^{r} \sum_{l=1}^{r} w_{k,l} u_{i+k-\lceil\frac{r}{2}\rceil,j+l-\lceil\frac{r}{2}\rceil} + O(h^q).
$$ We stipulate that $W$ be radially symmetric and set $w_{\lceil\frac{r}{2}\rceil,\lceil\frac{r}{2}\rceil}=0$, corresponding to the center value, which we \emph{assume} to be unknown and reconstructed. Then, setting $q = 4$ implies $r \geq 3$, but as noted in Remark \ref{rm:twobytwo}, stability considerations give $q \geq 5$. Among different choices, we set $w_{1,1}=w_{1,5}=w_{5,1}=w_{5,5} = 0$ and solve the underdetermined system
\begin{equation}
\begin{bmatrix}
8 & 4 & 4 & 4 \\
20 & 8 & 4 & 2
\end{bmatrix}
{\mathbf{w}}
=
\begin{bmatrix}
1 \\ 0
\end{bmatrix}.
\label{eq:weight}
\end{equation}
where ${\mathbf{w}} = [ w_{1,2}, w_{1,3}, w_{2,2}, w_{2,3}]^T$. To uniquely determine the weights, we solve the minimum norm problem
\begin{equation}
\min \| W \|_F^2
\label{eq:min}
\end{equation}
subject to \eqref{eq:weight}, where $\|\cdot \|_F$ denotes the Frobenius norm. A similar approach to stencil energy minimization is also used in \cite{coco2026efficient}. Although the convolution is applied only locally near the boundary, we establish its $L^2$-stability for completeness.

\begin{lemma}
Let $\lambda(k_x,k_y)$ denote the Fourier symbol given by
$$\sum_{-2 \leq j,l \leq 2} w_{j+3,l+3}e^{i(k_x j + k_y l)}$$
where $w_{i,j}$ solves (\ref{eq:min}) subject to (\ref{eq:weight}) and $w_{1,1}=w_{1,5}=w_{5,1}=w_{5,5} = 0$. Then,
$$
| \lambda(k_x,k_y) | \leq 1, \qquad (k_x,k_y)\in [-\pi,\pi]^2$$

\end{lemma}

\begin{proof}
Applying the standard method of Lagrange multiplier to the objective function $$\| W\|^2_{F}    = 8 (w_{1,2})^2 + 4(w_{1,3})^2+4(w_{2,2})^2+4(w_{2,3})^2$$ yields
$$
w_{1,2} = -\frac{7}{132}, w_{1,3} = \frac{1}{88}, w_{2,2} = \frac{37}{264}, w_{2,3} = \frac{9}{44}.
$$
One can then apply the change of variables $u = \cos(k_x)$ and $v = \cos(k_y)$ along with elementary trigonometric identities to show 
\[
\begin{aligned}
\left|\lambda(k_x,k_y)\right|
&=
\left| 2w_{2,3}(\cos (k_x) + \cos (k_y))
+ 4w_{2,2} \cos (k_x) \cos (k_y) \right.\\
&+ \left.2w_{1,3}(\cos (2k_x) + \cos (2k_y))
+ 4w_{1,2}\big(\cos(2k_x)\cos (k_y) + \cos (k_x) \cos(2k_y)\big) \right| \leq 1
% &  = \left| 2 w_{2,3} s t  + \right|
\end{aligned} 
\]
where the upper bound is attained when $k_x=k_y=0$, proving the claim.
\end{proof}

\begin{remark}
\label{rm:twobytwo}
The unique solution to 
\begin{equation*}
\begin{bmatrix}
4 & 4 \\
 4 & 2 \\
\end{bmatrix}
\begin{bmatrix}
 w_{2,2} \\
 w_{2,3} \\
\end{bmatrix}
=
\begin{bmatrix}
1 \\ 0 
\end{bmatrix}
\end{equation*}
yields the accuracy of $O(h^4)$, but it is no longer $L^2$-stable with $\max_{(k_x,k_y)\in [-\pi,\pi]^2}|\lambda(k_x,k_y)| = 3$.
\end{remark}

Now define
\[
N^h_{i,j}=\{(k,l)\in\mathbb{Z}^2:\ |k-i|\le2,\ |l-j|\le2,\ \text{and } w_{k-i+3,l-j+3}\neq0\},
\]
for each $(i,j)\in\mathbb{Z}^2$. Let $\Omega^h$ denote the set of all grid points in $\Omega\cup\partial\Omega$. We then define the reference zone by
\[
I^h=\{(i,j)\in\Omega^h:\ N^h_{i,j}\not\subset\Omega^h\},
\]
and the corresponding {refinement zone} by
\[
R^h=\{(i,j)\in\mathbb{Z}^2:\ (i,j)\in N^h_{p,q}\text{ for some }(p,q)\in I^h\}.
\]
Our goal is to compute a correction $\delta^h u$, supported on the refinement zone $R^h$, to the current approximation $u^h$. We formulate this as the constrained minimization problem
\begin{equation}
\min \|\delta^h u\|_{L^2}^2,
\label{eqn:min_norm}
\end{equation}
subject to
\begin{equation}
(W\star (u^h+\delta^h u))_{i,j}=u_{i,j},\qquad (i,j)\in I^h.
\label{eqn:correction}
\end{equation}

The width of the outer narrow band used to compute $u^h$ should be chosen sufficiently large to include the refinement zone, which depends only on $h$ and is independent of the exact values of $\phi$. For each $(i,j)\in I^h$, at least one grid point in $N^h_{i,j}$ belongs to the refinement zone, hence the constraint \eqref{eqn:correction} is generally underdetermined. Consistent with analytical considerations, we observe in all numerical experiments that the feasible set of the constraint is nonempty. Our definition of the reference zone is intended to incorporate as much interior information as possible into the system \eqref{eqn:correction}.

Given the current approximation $u^h = O(h^p)$, where $1 \leq p \leq 3$, the correction $\delta^h$ is expected to be of comparable magnitude, $\delta^h = O(h^p)$, since it is obtained as the minimum-norm correction satisfying the reconstruction constraints. Although our approach does not necessarily increase the order of accuracy of the corrected approximation, it is designed to avoid degrading the accuracy of the original approximation.

\section{Numerical Experiments}
\label{sec:3}
In all our experiments, numerical errors are measured in $L^\infty$ norm unless stated otherwise. The stopping criteria for the fast sweeping method is  
$$
\|u^h_{m+1}-u^h_{m}\|_\infty < 10^{-9}
$$
We solve the constrained minimization  (\ref{eqn:min_norm})-(\ref{eqn:correction}) by applying \textsc{matlab} built-in solver \texttt{lsqminnorm}. 

\subsection{Example 1}
As our first benchmark test case, we demonstrate that the field values in the \emph{single} first inner layer  can be constantly extended with third order accuracy. We consider $f(x,y) = \sin(\arctan(y/x))$ and the signed distance function $\phi(x,y) = \sqrt{x^2+y^2}-2$ over the computational domain $[-\pi,\pi]^2$. The accuracy of the constant extrapolation is reported in Table \ref{tb:ex1}.

\begin{table}[htbp]
\centering
\begin{tabular}{|c|c|c|c|c|}
\hline
\multicolumn{1}{|c|}{Mesh} 
& \multicolumn{1}{c|}{$101\times101$}
& \multicolumn{1}{c|}{$201\times201$}
& \multicolumn{1}{c|}{$401\times401$}
& \multicolumn{1}{c|}{$801\times801$} \\
\hline
Error & 5.240E-5 & 6.603E-6 & 9.292E-7 & 1.184E-7 \\
\hline
Conv.\ order & --- & 2.988 & 2.829 & 2.973 \\
\hline
\end{tabular}
\caption{Numerical accuracy of constant extrapolation in a narrow band of width 
$3h$ for Example 1}
\label{tb:ex1}
\end{table}

\subsection{Example 2} 

We revisit Example~1 of \cite{aslam2004partial}, using the same setup as in Example~1, except that we now consider the function
$f(x,y) = \sin(x)\cos(y)$. We compute constant, linear, and quadratic extrapolations, where the required normal derivatives are computed and stored only on the first inner  layer. The results are reported in Table~\ref{tb:ex2}, and demonstrate convergence rates consistent with the expected order of accuracy. We further note that the number of iterations remains uniformly bounded under mesh refinement. The reported CPU times (in seconds) correspond to a MATLAB implementation on a personal laptop and are comparable with those reported in Table~2 of \cite{aslam2014static}.

Next, we examine the accuracy improvement achieved by the proposed reconstruction technique, as summarized in Table~\ref{tb:ex2_hfilter}. With a very small increase in computational cost, the resulting errors are reduced significantly, leading to an increase in the observed order of accuracy.

\begin{table}[!t]
\centering
\begin{tabular}{|c|c|c|c|c|}
\hline
\multicolumn{1}{|c|}{Mesh} 
& \multicolumn{1}{c|}{$201\times201$}
& \multicolumn{1}{c|}{$401\times401$}
& \multicolumn{1}{c|}{$801\times801$}
& \multicolumn{1}{c|}{$1601\times1601$} \\
\hline
\multicolumn{5}{|c|}{{Constant extrapolation}} \\
\hline
$L^\infty$ error & 8.072E-2 & 3.999E-2 & 2.031E-02 & 1.040E-02 \\
\hline
Conv.\ order & --- & 1.013 & 0.978 & 0.966 \\
\hline
\# Iter & 10 & 9 & 9 & 9 \\
\hline
CPU time & 0.0168 & 0.0314 & 0.0752 & 0.222 \\
\hline
\multicolumn{5}{|c|}{{Linear extrapolation}} \\
\hline
$L^\infty$ error & 3.327E-3 & 1.008E-3 & 3.350E-4 & 8.387E-5 \\
\hline
Conv.\ order & --- & 1.723 & 1.589 & 1.998 \\
\hline
\#  Iter & 10, 9 & 10, 8 & 9, 7 & 9, 6 \\
\hline
CPU time & 0.0300 & 0.0572 & 0.138 & 0.358 \\
\hline
\multicolumn{5}{|c|}{{Quadratic extrapolation}} \\
\hline
$L^\infty$ error & 7.986E-4 & 1.006E-4 & 1.262E-5 & 1.580E-6 \\
\hline
Conv.\ order & --- & 2.989 & 2.994 & 2.997 \\
\hline
\#  Iter & 10, 9, 8 & 10, 8, 8 & 9, 7, 7 & 9, 6, 6 \\
\hline
CPU time & 0.0507 & 0.111 & 0.265 & 0.762 \\
\hline
\end{tabular}
\caption{Numerical accuracy of constant, linear and quadratic extrapolation in the narrow band of size $3h$ for Example 2}

\label{tb:ex2}
\end{table}

\begin{table}[!t]
\centering
\begin{tabular}{|c|c|c|c|c|c|}
\hline
Mesh 
& $200\times200$ 
& $400\times400$ 
& $800\times800$ 
& $1600\times1600$ 
& $3200\times 3200$\\
\hline

\multicolumn{6}{|c|}{Constant Extrapolation} \\
\hline
$L^\infty$ error & 2.409E-2 & 1.064E-2 & 3.829E-3 & 1.052E-3 & 2.674E-4 \\
\hline
 Conv. order &  ---& 1.179 & 1.474 & 1.866 & 1.974 \\
\hline
{Extra CPU time} & 4.950E-4 & 5.657E-4 & 8.118E-4 & 1.173E-3 & 1.627E-3 \\
\hline
\multicolumn{6}{|c|}{Linear Extrapolation} \\
\hline
$L^\infty$ error & 9.167E-4 & 2.870E-4 & 2.792E-5 & 1.155E-5 & 1.322E-6 \\
\hline
 Conv. order &  ---& 1.676 & 3.362 & 1.273 & 3.127 \\
\hline
{Extra CPU time} & 5.382E-4 & 1.780E-3 & 1.350E-3& 2.835E-3 & 3.104E-3 \\
\hline
\multicolumn{6}{|c|}{Quadratic Extrapolation} \\
\hline
$L^\infty$ error & 1.579E-4 & 1.876E-5 & 1.810E-6 & 1.280E-7 & 8.147E-9 \\
\hline
 Conv. order &  ---& 3.073 & 3.373 & 3.823 & 3.973 \\
\hline
{Extra CPU time} & 4.521E-3 & 2.225E-3 & 2.650E-3 & 2.181E-3 & 2.440E-3 \\
\hline
\end{tabular}
\caption{Numerical Errors after applying the reconstruction technique in the refinement zone for Example 2}
\label{tb:ex2_hfilter}
\end{table}

\subsection{Example 3}
We follow \cite{moroney2017extending} and consider the union of two circular domains whose boundary is defined as the zero level set of
\[
\phi(x, y) = \min\left(\sqrt{(x-0.8)^2 + y^2}-1,,
\sqrt{(x+0.8)^2 + y^2}-1\right).
\]
Table 2 of their work reports that, despite the presence of kinks, linear and quadratic extrapolations based on the time dependent approach of \cite{aslam2004partial} exhibit second and third order convergence, respectively. In our static approach, we recover the same convergence rates when the normal vector $\mathbf{n}$ is not  normalized numerically. With the normalized $\mathbf{n}$, the convergence reduces by one order in both cases, but the proposed boundary reconstruction restores the expected convergence rates (Figure \ref{fig:ex3_yes_filter}).

\begin{figure}[!t]
    \centering
    \includegraphics[width=0.6\linewidth]{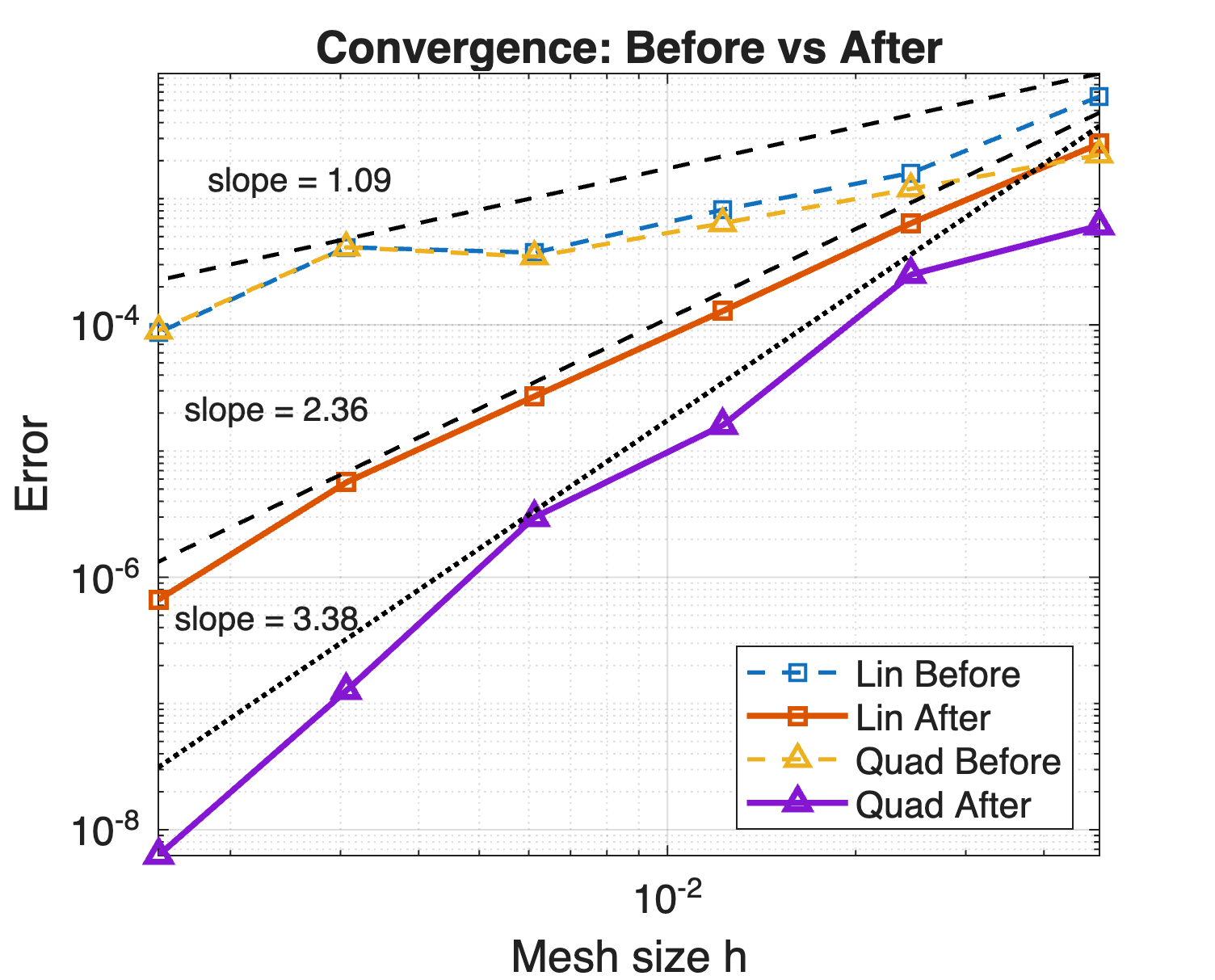} 
    \caption{Comparison of convergence rates of linear and quadratic extrapolation for Example 3 in the refinement zone before and after reconstruction}
    \label{fig:ex3_yes_filter}
\end{figure}

\subsection{Example 4}
We further demonstrate the robustness of our reconstruction technique by considering the numerical experiment from \cite{bochkov2020pde}. Similarly to Example 3, we consider the domain defined as the union of two circular regions within the unit square $[0,1]^2$, given by
$$
\phi(x, y) = \min\left(\sqrt{(x+0.1)^2 + (y+0.3)^2}-0.501,\;
                   \sqrt{(x-0.2)^2 + (y-0.2)^2}-0.401\right).
$$

As shown in Figure \ref{fig:ex4_convergence_ref}, both the linear and quadratic extrapolations, when applied as standalone methods, exhibit approximately first order convergence. In contrast, the proposed reconstruction technique restores the expected convergence rates in both cases, despite the similar magnitude of the pre-processed errors. Compared with the weighted Cartesian approach of \cite{bochkov2020pde}, our approach achieves comparable accuracy at a reduced computational cost, as it does not require tensor valued quantities.

\begin{figure}[!t]
    \centering
    \includegraphics[width=0.6\linewidth]{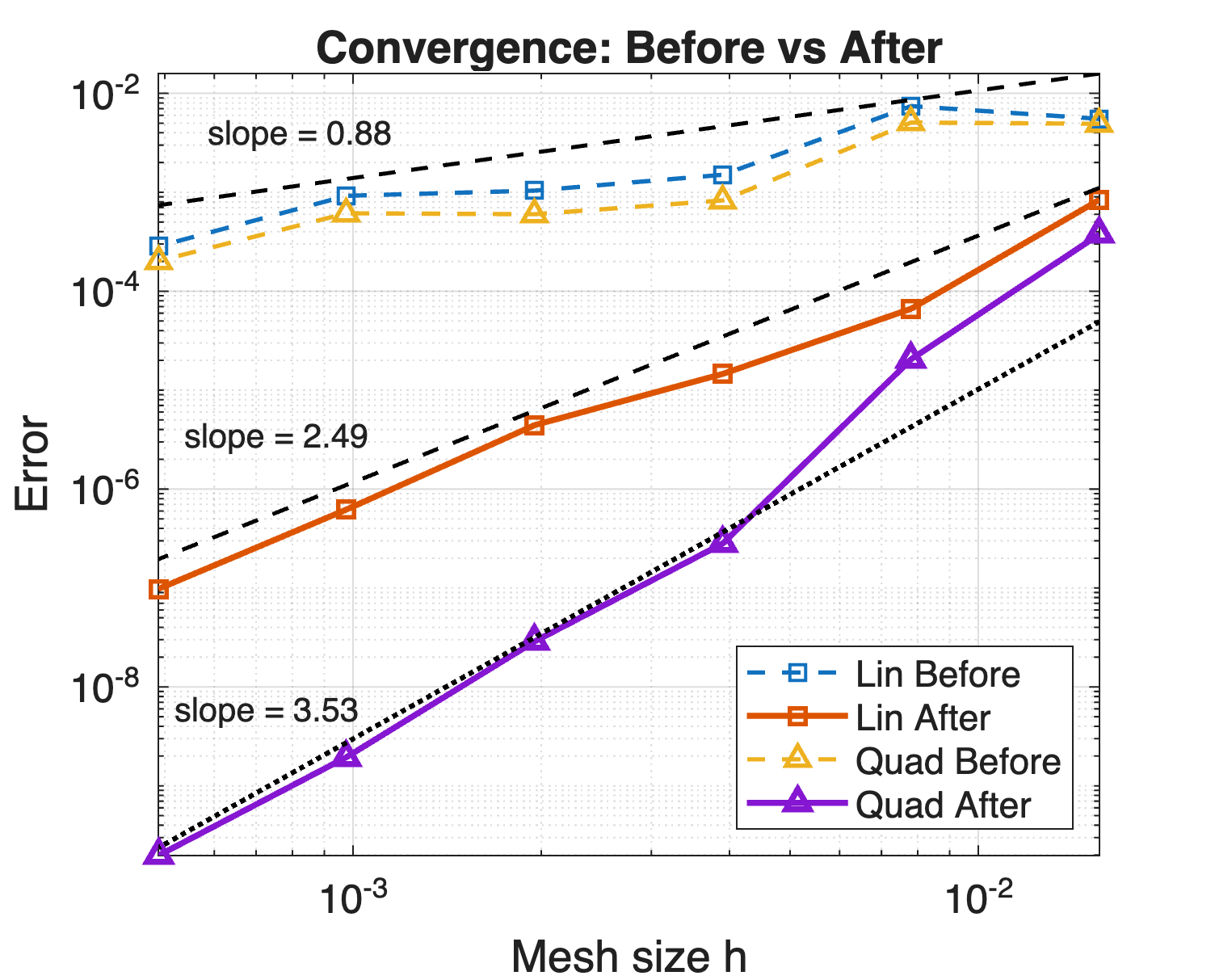} 
    \caption{Comparison of convergence rates for Example 4 in the refinement zone before and after applying the reconstruction technique}
    \label{fig:ex4_convergence_ref}
\end{figure}

\subsection{Example 5}
As our final numerical test, we consider a moon shaped domain defined by
$$
\phi(x,y) = \max\left(\sqrt{x^2 + y^2}-0.501,
-(\sqrt{(x-0.4)^2 + (y-0.3)^2}-0.401)\right),
$$
following the construction in \cite{bochkov2020pde}. The boundary of the domain contains two kinks at the intersections of its convex and concave segments. Near these kinks, some of the upwind neighbors of first inner layer grid points lie outside $\Omega$, reducing the accuracy of the local least squares fitting (Figure~\ref{fig:ex5_stencil}). As illustrated in the same figure, this issue can be addressed by relaxing the upwind neighborhood of $(i,j)$ to include all grid points belonging to the symmetric stencil
$$
\{(k,l): |k-i| \leq 3,\ |l-j| \leq 3\}
$$
that also satisfy the upwind condition
\[ [x(k,l)-x(i,j),\,y(k,l)-y(i,j)] \cdot [\phi_x(i,j),\,\phi_y(i,j)] \leq 0 \]

\begin{figure}[!t]
\centering
\includegraphics[width=0.495\textwidth]{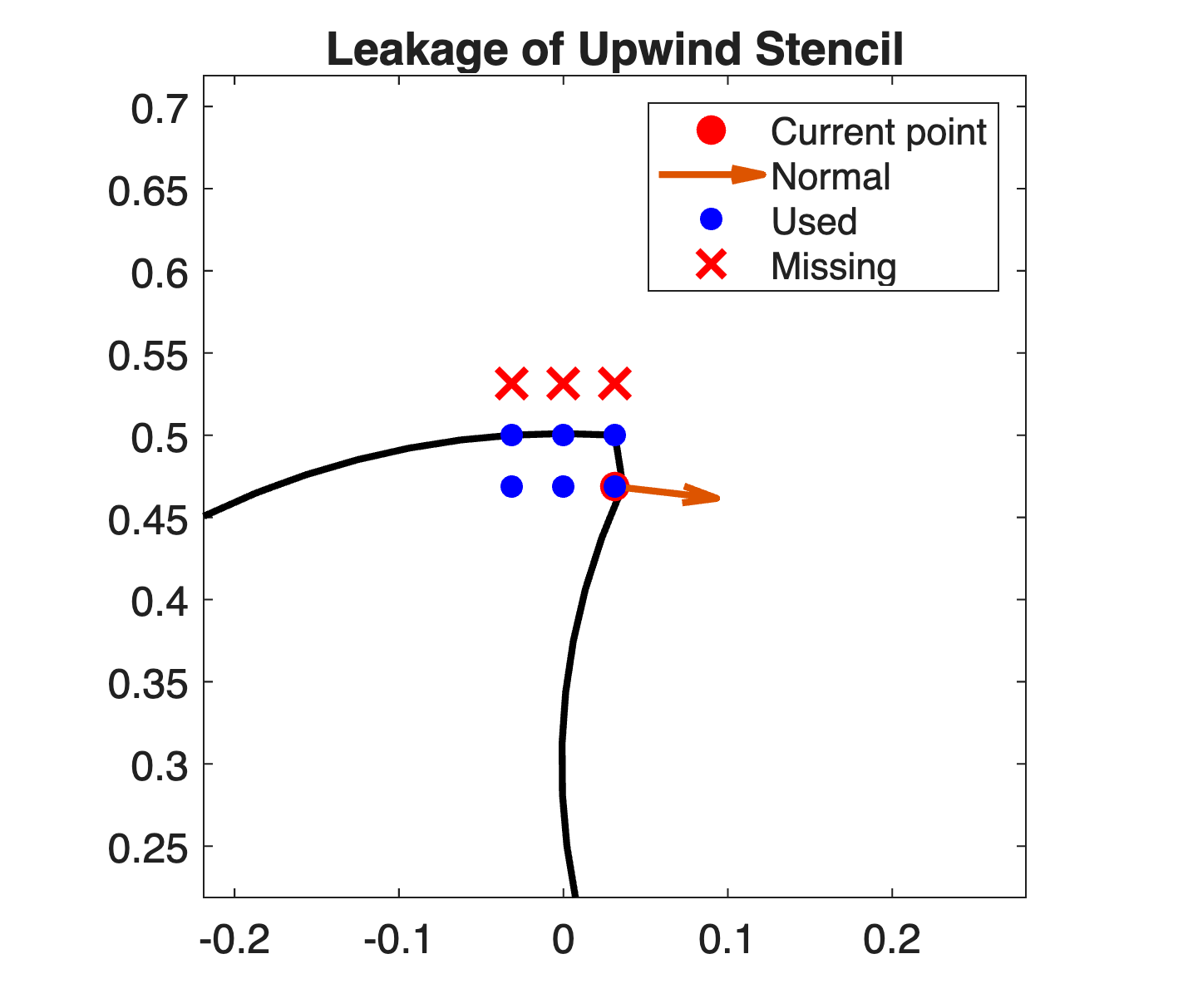}
\hspace{-0.01\textwidth}
\includegraphics[width=0.495\textwidth]{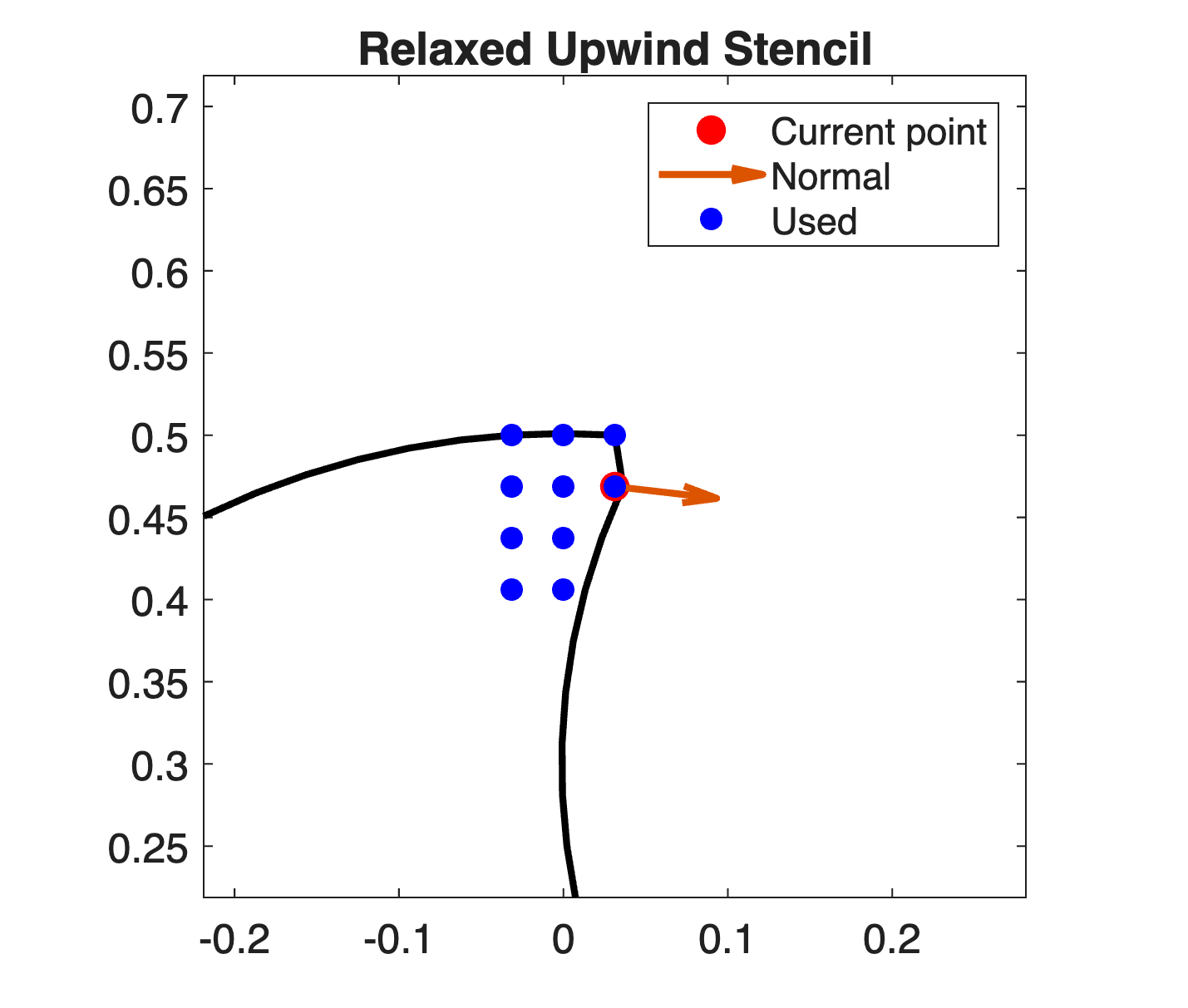}
\caption{Illustration of missing upwind neighbors (left) and the relaxed upwind stencil (right) for the first inner layer grid point $\left(\frac{1}{32}, \frac{15}{32}\right)$ in Example~5.}
\label{fig:ex5_stencil}
\end{figure}

Figure \ref{fig:ex5_convergence} shows that both the linear and quadratic extrapolation methods achieve first order convergence; our reconstruction technique is only accuracy-preserving. This can be attributed to the strong singularities at the kinks, where the normal derivative is discontinuous and the curvature changes sign. For comparison, the figure also includes the weighted Cartesian derivative-based extrapolation of \cite{bochkov2020pde}, which recovers the expected first and second order convergence rates for the linear and quadratic extrapolation methods, respectively.

A closer examination of the effect of our reconstruction technique for the weighted Cartesian approach is provided in Figure \ref{fig:ex5_refinement_analysis}. The $L^\infty$ error is reduced by at least $30\%$. More notably, we report the improved convergence rate of the $L^2$ error normalized by the number of grid points in the refinement zone. The increase in order by $0.5$ for both the linear and quadratic cases indicates a significant reduction in errors away from the kinks after applying the reconstruction technique. We emphasize that these improvements are achieved with only small additional computational cost.

\begin{figure}[!t]
    \centering
    \includegraphics[width=0.493\textwidth]{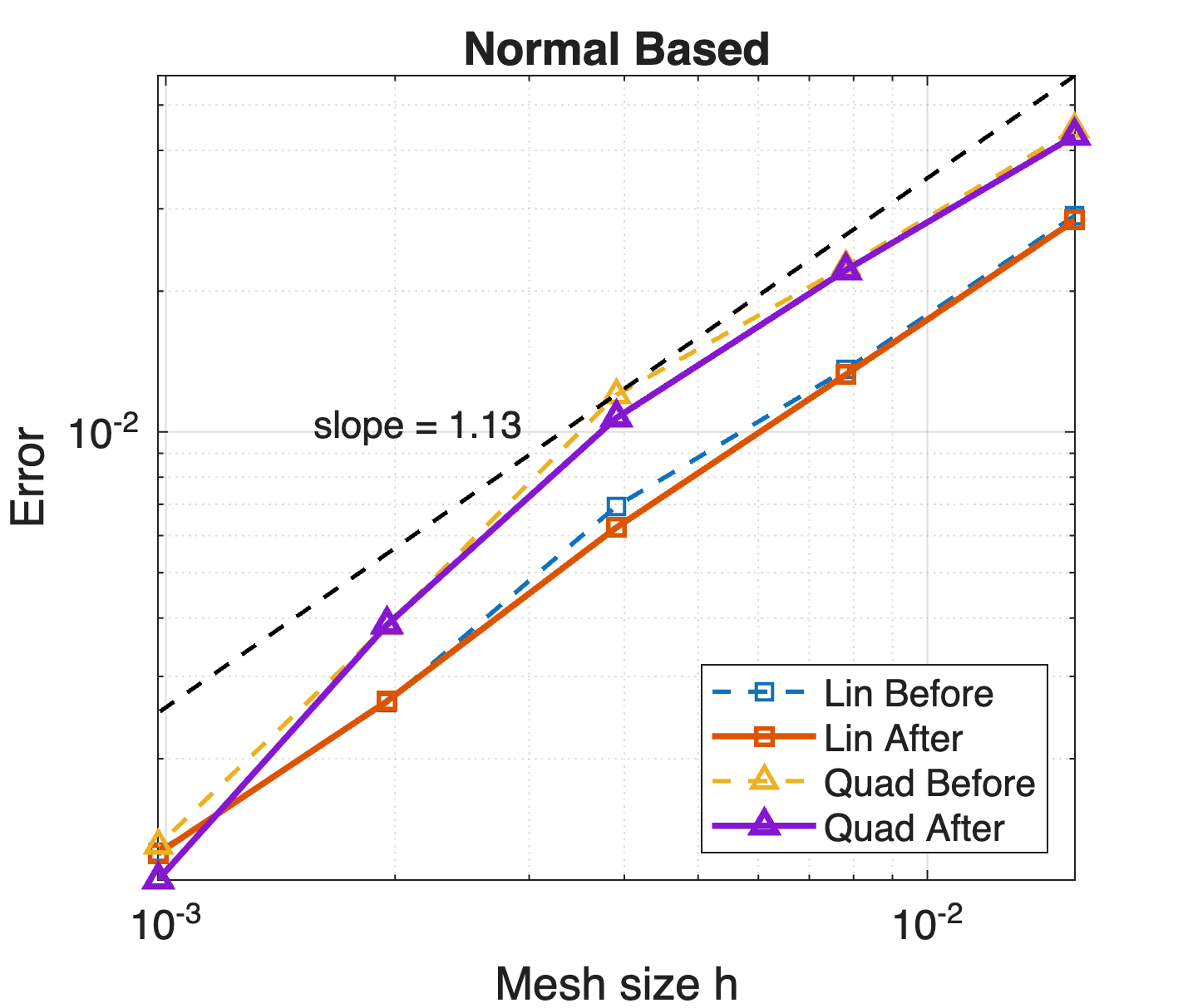}
    \hspace{-0.001\textwidth}
    \includegraphics[width=0.493\textwidth]{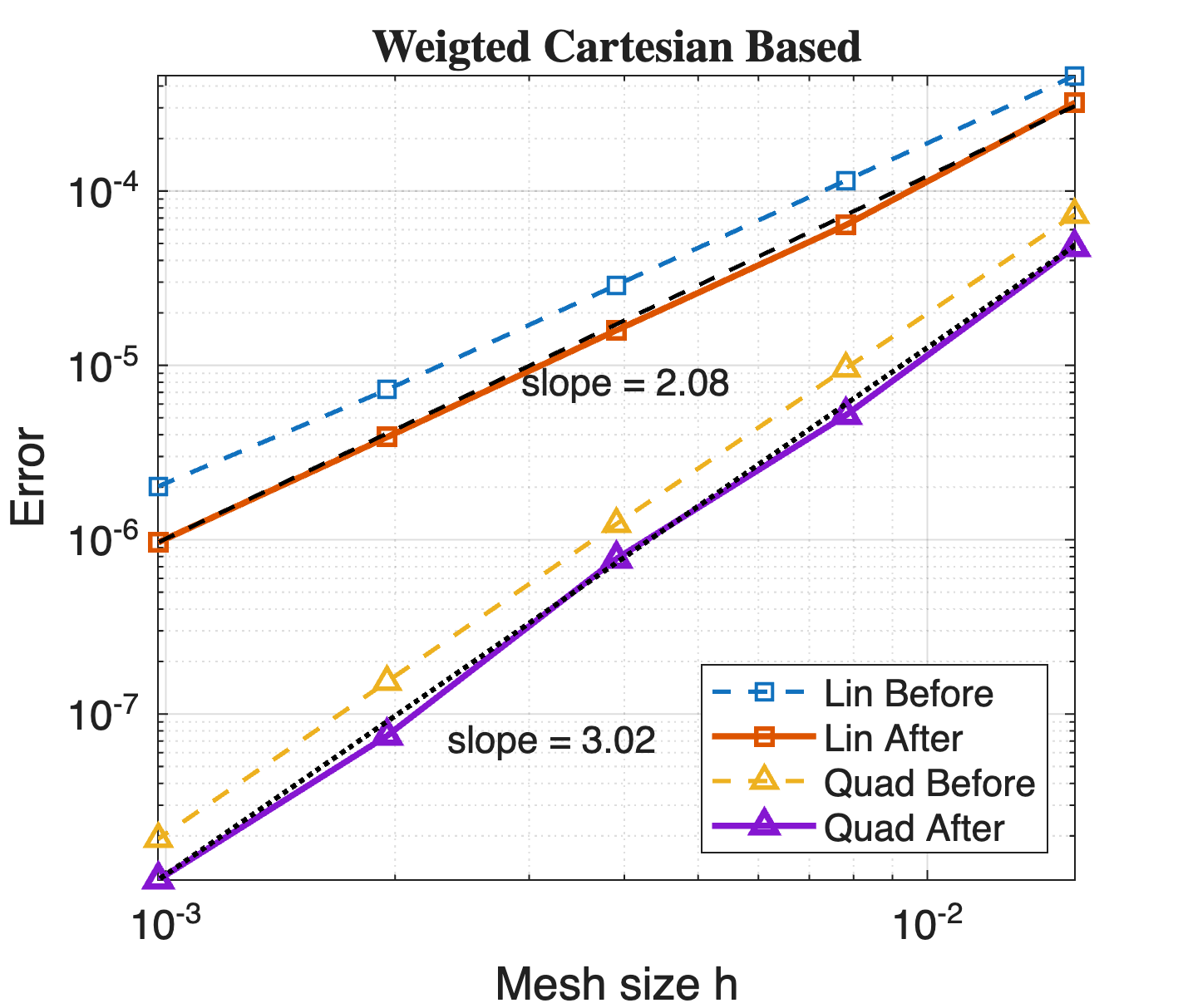}
    \caption{Convergence plots for Example 5 in the refinement zone \textbf{Left:} normal derivatives-based extension by \cite{aslam2014static}. \textbf{Right:} weighted Cartesian derivatives-based extension by \cite{bochkov2020pde}.}
    \label{fig:ex5_convergence}
\end{figure}

\begin{figure}[!t]
    \centering
    \includegraphics[width=0.493\textwidth]{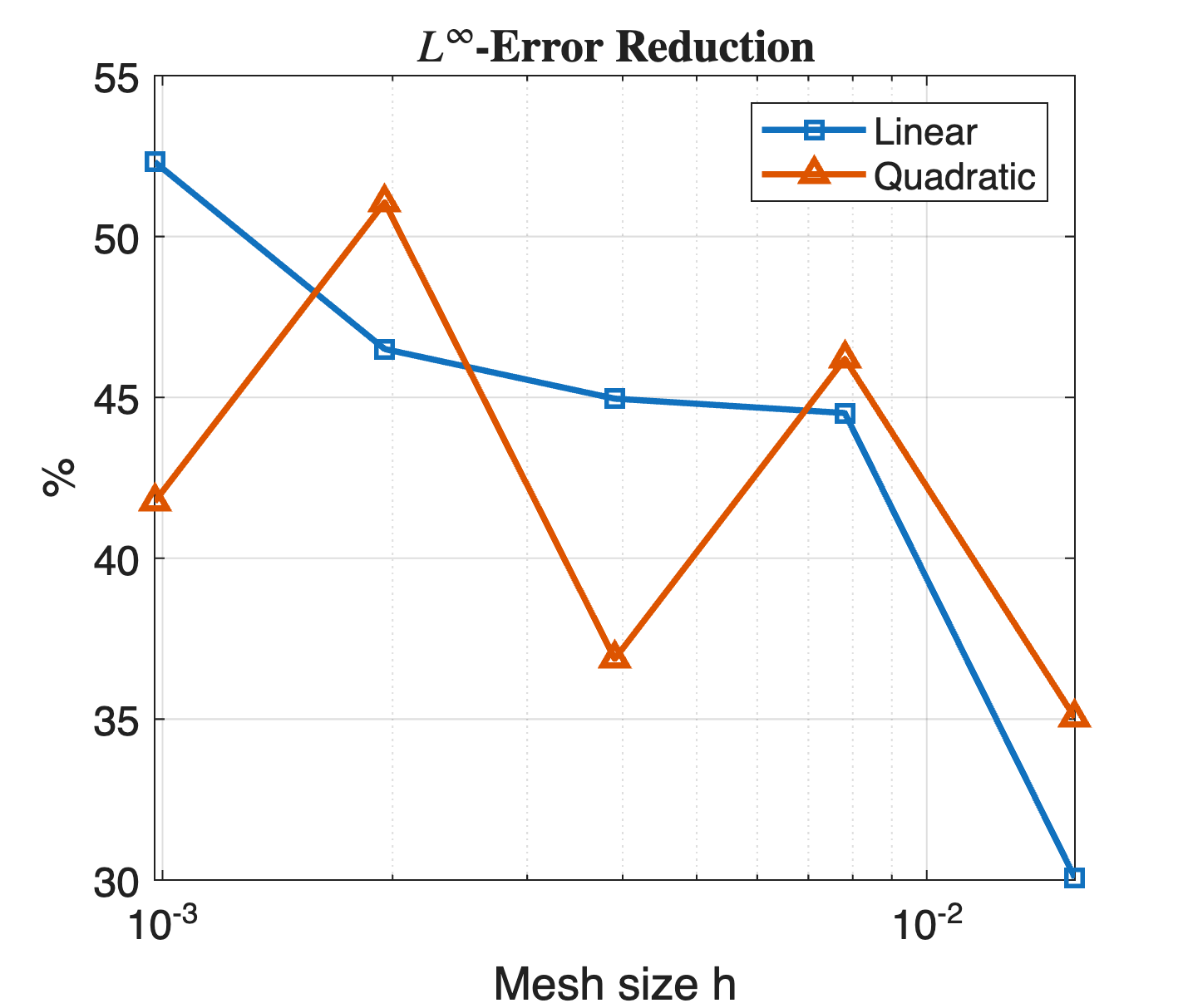}
    \hspace{-0.001\textwidth}
    \includegraphics[width=0.493\textwidth]{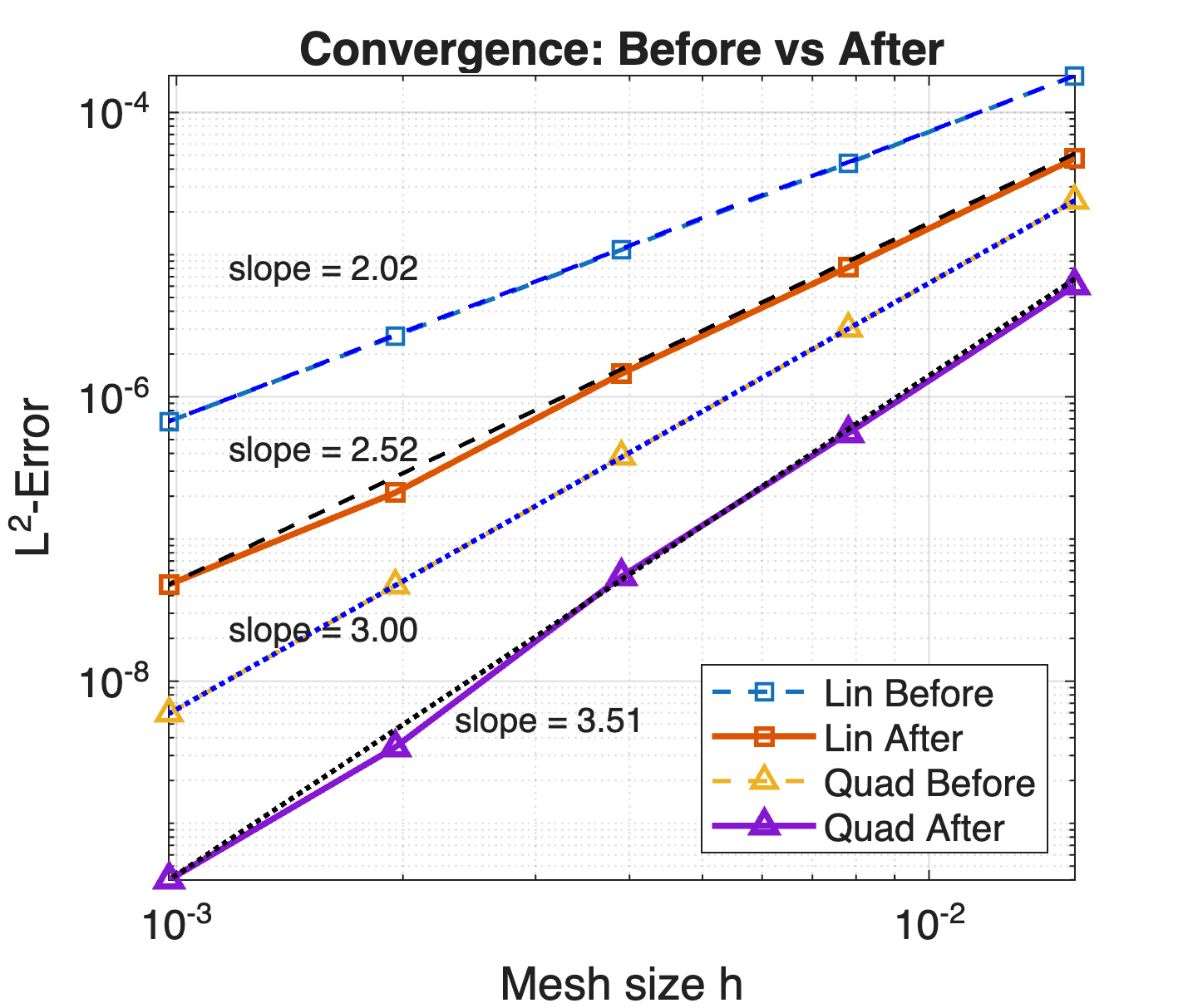}
    \caption{Error Reduction by the reconstruction technique in the refinement zone for Example 5. }
    \label{fig:ex5_refinement_analysis}
\end{figure}

\section{Conclusion}
\label{sec:4}

We propose two techniques, relaxed upwinding finite differences and a boundary reconstruction technique, to enhance the efficiency and accuracy of PDE based extrapolation. Both techniques are inspired by introducing an additional degree of freedom, which provides stability while maintaining the desired accuracy. In addition to their simplicity, the proposed techniques eliminate the need to compute and store normal derivative data on computational nodes that are not immediately adjacent to the interface. The effectiveness of the resulting extrapolation approach is demonstrated through numerical examples involving both regular and irregular domain boundary geometries.

\section*{Acknowledgment} The author would like to thank Yingjie Liu for carefully reading this paper and providing valuable comments.
%%%%%%%%%%%
%%%%%%%%%%%
\newpage
\clearpage
\bibliographystyle{plain}
\addcontentsline{toc}{section}{\refname}
\bibliography{references}

\end{document}